\documentclass[reqno]{amsart}
\usepackage{amssymb}
\usepackage{graphicx}

\usepackage[usenames, dvipsnames]{color}
\usepackage{verbatim}
\usepackage{mathrsfs}

\numberwithin{equation}{section}

\newtheorem{theorem}{Theorem}[section]

\newtheorem{lemma}[theorem]{Lemma}

\theoremstyle{definition}
\newtheorem{remark}[theorem]{Remark}

\theoremstyle{definition}

\theoremstyle{definition}

\makeatletter
\def\dashint{\operatorname%
{\,\,\text{\bf-}\kern-.98em\DOTSI\intop\ilimits@\!\!}}
\makeatother

\def\\det{\text{det}}

\def\.5{\frac{1}{2}}

\newcommand{\RN}[1]{%
  \textup{\uppercase\expandafter{\romannumeral#1}}%
}

\renewcommand{\epsilon}{\varepsilon}

\newcounter{marnote}

\begin{document}
\title[Asymptotic behaviors]{Optimal Asymptotic Behavior at Infinity of Ancient Solution to the Parabolic Monge-Ampère Equation with Slow Perturbation Term}
\author[K. Yan]{Kui Yan}
\address[K. Yan]{School of Mathematical Sciences, Beijing Normal University, Laboratory of Mathematics and Complex Systems, Ministry of Education, Beijing 100875, China.}
\email{202231130028@mail.bnu.edu.cn}

\author[J. Bao]{Jiguang Bao\textsuperscript{*}}
\address[J. Bao]{School of Mathematical Sciences, Beijing Normal University, Laboratory of Mathematics and Complex Systems, Ministry of Education, Beijing 100875, China.}
\email{jgbao@bnu.edu.cn}

\footnotetext[1]{corresponding author}

\date{\today} 

\begin{abstract}
In this paper, we obtain optimal asymptotic behavior of parabolically convex $C^{2,1}$ solution to the parabolic Monge-Ampère equation $-u_t\det D_x^2u=f$, where $f$ converges to $1$ at infinity with a slow rate. This result extends the elliptic estimate in \cite{lb5} to the parabolic setting.

\end{abstract}

\maketitle

\section{introduction and main results} \label{introduction and main results}

The Monge-Ampère equation plays a fundamental role in diverse areas of mathematics, including differential geometry (prescribed Gaussian curvature problems, affine geometry) and applied analysis (optimal transport theory).  This fully nonlinear partial differential equation has found significant applications across multiple domains of geometric analysis and nonlinear PDE theory. 

The classical result from Jörgens-Calabi-Pogorelov \cite{j54,c58,p72} asserts that convex entire $C^2$ solution of $\det D_x^2u=1$ in $\mathbb{R}^n$ must be a quadratic polynomial. Cheng-Yau \cite{cy86} provided a simpler and more analytical proof. Caffarelli \cite{c95} generalized the classical JCP theorem to the frame of viscosity solutions. Jost-Xin \cite{jx01} developed an alternative proof. 

Caffarelli-Li \cite{cl03} obtained asymptotic behaviors of solutions to  $\det D_x^2u=f$ in $\mathbb{R}^n$, when $\text{supp}(f-1)$ is bounded in $\mathbb{R}^n$. Bao-Li-Zhang \cite{blz15} generalized Caffarelli-Li's results to the case where $f\in C^{m}(m\ge3)$ outside a bounded set in $\mathbb{R}^n$ and satisfies that for $\beta>2$, $i\in\mathbb{N}$ and $i\le m$
\[
\limsup_{|x|\rightarrow\infty}|x|^{\beta+i}\left|D^i_x(f(x)-1)\right|<\infty.
\]
There are also many extensive studies on the asymptotic behaviors of solutions, see \cite{fmm99,h20,lb1,lb2,lb3,lb4}. 

For $\beta\in(0,2]$, Liu-Bao \cite{lb5} proved slow convergence behaviors. More precisely, there exist a $n\times n$ positive definite matrix $A$ with $\det A=1$, $b\in\mathbb{R}^n$, such that for $i\in\mathbb{N}$ and $i\le m+1$
\[
\limsup_{|x|\rightarrow\infty}|x|^{\beta-2+i}\left|D_x^i\left( u-p_\beta\right)(x)\right| <\infty,\quad \beta\neq1,2,
\]
\[
\limsup_{|x|\rightarrow\infty}|x|^{\beta-2+i}\left(\ln |x|\right)^{-1}\left|D_x^i\left( u-p_\beta\right)(x)\right| <\infty,\quad \beta=1,2,
\]
where
\[
p_\beta(x):=\frac{1}{2}x'Ax+
\left\{\begin{array}{ll}
b\cdot x, & \beta\in(1,2], \\
0, & \beta\in(0,1].
\end{array}\right. 
\]
The results in \cite{lb5} for $\beta\neq1$ are optimal, while the optimality for $\beta=1$ is unknowm.

Krylov \cite{k76} firstly introduced the parabolic Monge-Ampère equations
\begin{equation}\label{pma1}
-u_t\det D_x^2u=f,
\end{equation}
which are closely relevant to Aleksandrov-Bakelman-Pucci type maximal principle for parabolic equations \cite{t85ab}, controlled diffusion process \cite{k81a,k81b} and  deformation of a surface related to Gauss–Kronecker curvature \cite{t85}. 

For reader's convenience, we introduce some notations here. Let $\mathbb{R}^{n+1}_-:=\mathbb{R}^n\times(-\infty,0]$. A function $u(x,t)$ defined in $\mathbb{R}^{n+1}_-$ is called parabolically convex if it is convex in $x$ and nonincreasing in $t$. We use $u_t$ or $D_tu$ to denote $u$'s derivative with respect to $t$ variable. We denote $D_x^i D_t^j u$ as $u$'s mixed derivative, which is $i\text{-th}$ order with respect to $x$ variable and $j\text{-th}$ order with respect to $t$ variable. We say that $f$ is $C^{2 k, k}(k\in\mathbb{N}^+)$ if $D_x^i D_t^j f$ is continous for $i,j\in\mathbb{N}$ and $i+2 j \leq 2 k$. $\mathcal{R}:=\left(|x|^2-t\right)^{\frac{1}{2}}$ will be frequently used.

Guti\'errez-Huang \cite{gh98} got that, when there exist $m_{1},m_{2}>0$ such that
\begin{equation}\label{ut}
m_{1}\le -u_t\le m_{2}\quad \text{in} \quad \mathbb{R}^{n+1}_-
\end{equation}
and $f\equiv1$, the parabolically convex $C^{4,2}$ ancient solution to \eqref{pma1} must be of form $-\tau t+p(x)$, where $\tau>0$ and $p(x)$ is a convex quadratic polynomial. Xiong-Bao \cite{xb11} extended Guti\'errez-Huang's results to more general equations. 

Zhang-Bao-Wang \cite{zbw18} obtained asymptotic behaviors of $C^{2,1}$ ancient solution when the support of $f-1$ is bounded and $n\ge3$. To be concrete, denote
\[
t_0:=\inf\{t\le0:(x,t)\in\text{supp}\left(f-1\right)\},
\]
then $u\equiv -\tau t+p(x)$ in $\mathbb{R}^n\times\left(-\infty,t_0\right]$ and $u$ converges to $-\tau t+p(x)$ at rate $\mathcal{R}^{2-n}$ in $\mathbb{R}^n\times\left[t_0,0\right]$. For $n\ge1$, Yan-Bao \cite{yb25} derived a convergence result when $f\in C^{2m,m}(m\ge2)$ outside a bounded set in $\mathbb{R}^{n+1}_-$ and satisfies that for $i,j\in\mathbb{N}$ and $i+2j\le 2m$
\begin{equation}\label{f}
\limsup_{\mathcal{R}\rightarrow\infty}\mathcal{R}^{\beta+i+2j}\left|D^i_xD^j_t(f(x,t)-1)\right|<\infty.
\end{equation}
They proved that there exist a positive definite matrix $A$, $b\in\mathbb{R}^n$, $c\in\mathbb{R}^1$, $\tau>0$ with $\tau\det A=1$ such that for $p(x):=\frac{1}{2}x'Ax+b\cdot x+c$, $i,j\in\mathbb{N}$ and $i+2j\le 2m$
\[
\limsup_{\mathcal{R}\rightarrow\infty}\mathcal{R}^{\beta-2+i+2j}\left|D_x^iD_t^j\left( u(x,t)-p(x)+\tau t\right)\right| <\infty.
\]
When $\text{supp}(f-1)$ is bounded, a convergence result at rate $\mathcal{R}^{-s}$ for any $s>0$ in the region $\mathbb{R}^n\times\left(-\infty,t_0\right]$ was also established.

In this paper, we study asymptotic behavior of solutions when $f$ satisfies \eqref{f} for $m\ge3$ and $\beta\in(0,2]$. 

\begin{theorem}\label{mainthm1}
For $n\ge1$, let $u\in C^{2,1}(\mathbb{R}^{n+1}_{-})$ be a parabolically convex solution to \eqref{pma1} with \eqref{ut}, where $f\in C^0(\mathbb{R}^{n+1}_{-})$ satisfies \eqref{f} for $m\ge3$ and $\beta\in(0,2]$. Then $u\in C^{2m,m}$ outside a bounded set in $\mathbb{R}^{n+1}_{-}$. Additionally there exist a $n\times n$ symmetric positive definite matrix $A$, $b\in\mathbb{R}^n$ and $\tau>0$ with $\tau\det A=1$ such that for $i,j\in\mathbb{N}$ and $i+2j\le 2m$
\[
\limsup_{\mathcal{R}\rightarrow\infty}\mathcal{R}^{\beta-2+i+2j}\left|D_x^iD_t^j\left( u-p_\beta\right)\right| <\infty,\quad \beta\neq1,2,
\]
\[
\limsup_{\mathcal{R}\rightarrow\infty}\mathcal{R}^{\beta-2+i+2j}\left(\ln \mathcal{R}\right)^{-1}\left|D_x^iD_t^j\left( u-p_\beta\right)\right| <\infty,\quad \beta=1,2,
\]
where
\[
p_\beta(x,t):=-\tau t+\frac{1}{2}x'Ax+
\left\{\begin{array}{ll}
b\cdot x, & \beta\in(1,2], \\
0, & \beta\in(0,1].
\end{array}\right. 
\]
\end{theorem}

\begin{remark}\label{best}
The convergence rate established in Theorem \ref{mainthm1} is sharp for all $\beta\in\left(0,2\right]$. 

To demonstrate this, we construct an explicit smooth function in $\mathbb{R}^{n+1}_-$ as follows:
\[
u(x,t):=\frac{1}{2}|x|^2-t+\delta g_\beta(x,t),\quad (x,t)\in\mathbb{R}^{n+1}_-
\]
where $\delta>0$ is a parameter, and $g_\beta\in C^{\infty}(\mathbb{R}^{n+1}_-)$ satisfies that
\begin{equation}\label{optimal}
g_\beta(x,t)=
\left\{\begin{array}{ll}
\mathcal{R}^{2-\beta} & \text { for } \beta\ne1,2, \\
\sum\limits_{i=1}^nx_i\ln\mathcal{R} & \text { for } \beta=1, \\
\ln\mathcal{R} & \text { for } \beta=2,
\end{array}\right.\quad \text { for } \mathcal{R}>2, 
\end{equation}
with $g_\beta(x,t)=0$ for $0<\mathcal{R}<1$, and the derivative bounds
\[
|g_\beta'|+|g_\beta''|\le C\quad \text{for}\quad 1<\mathcal{R}<2,
\]
where $C>0$ is a constant. A direct computation shows that
\[
D_x^2u>\frac{1}{2}I\quad \text{and} \quad \frac{1}{2}\le-u_t\le\frac{3}{2}
\] 
for sufficiently small $\delta$. Consequently,  $u$ is parabolically convex in $\mathbb{R}^{n+1}_-$. For $\mathcal{R}$ sufficiently large, we derive the asymptotic expansion
\[
\begin{aligned}
-u_t\det D_x^2u
&=1+\delta\left(2-\beta\right)\left(n+\frac{1}{2}\right)\mathcal{R}^{-\beta}+\delta\left(\beta-2\right)\beta\mathcal{R}^{-\beta-2}|x|^2 + O\left(\mathcal{R}^{-2\beta}\right)\\
&=1+O\left(\mathcal{R}^{-\beta}\right),
\end{aligned}
\]
which confirms the ‌optimality‌ of the estimates in Theorem \ref{mainthm1} for $\beta\ne1,2$. For $\beta=2$, there is $u(x,t)=\frac{1}{2}|x|^2-t+O\left(\ln\mathcal{R}\right)$ and
\[
-u_t\det D_x^2u = 1+\delta\left(n+\frac{1}{2}\right)\mathcal{R}^{-2}-2\delta|x|^2\mathcal{R}^{-4}+O\left(\mathcal{R}^{-4}\right)=1+O\left(\mathcal{R}^{-2}\right),
\]
validating the sharpness of the estimate for $\beta=2$. For $\beta=1$, notice that
\[
u(x,t)=\frac{1}{2}|x|^2-t+O\left(\mathcal{R}\ln\mathcal{R}\right)
\]
and
\[
\begin{aligned}
-u_t\det D_x^2u&=1+\delta\mathcal{R}^{-2}\left(\sum\limits_{i=1}^{n}x_i\right)\left(n+2-2\mathcal{R}^{-2}|x|^2\right)+O\left(\mathcal{R}^{-2}\right)\\
&=1+O\left(\mathcal{R}^{-1}\right),
\end{aligned}
\]
which confirms the sharpness of the estimate for $\beta=1$.
\end{remark}

\begin{remark}
The asymptotic estimate for the elliptic Monge-Ampère equation when $\beta=1$ in \cite{lb5} is in fact sharp. To demonstrate this optimality, we construct an example as follows. Consider a smooth function $g\in C^{\infty}(\mathbb{R}^n)$ satisfying
\begin{equation}\label{beta=2}
g(x)=\left\{\begin{array}{ll}
\sum\limits_{i=1}^nx_i\ln|x| & \text { for } |x|>2, \\
0 & \text { for } 0<|x|<1.
\end{array}\right.
\end{equation}
Define the function
\[
u(x):=\frac{1}{2}|x|^2+\delta g\left(x\right),\quad x\in\mathbb{R}^n
\]
for some small $\delta>0$. A direct computation yields
\[
\det D^2_xu=1+\delta n|x|^{-2}\sum\limits_{i=1}^nx_i+O\left(|x|^{-2}\right)=1+O\left(|x|^{-1}\right),
\]
which establishes the optimality for $\beta=1$.
\end{remark}

This paper is structured as follows. In the remainder of this section, we introduce key notations. Section \ref{a slow convergence solution} presents the construction of a slow-converging solution to the heat equation. Next, we give the smoothing process of the solution to the parabolic Monge-Ampère equation in Section \ref{smoothing of solution}. Finally, Sections \ref{proof of Theorem 1} contains the proofs of Theorems \ref{mainthm1}.

For $\delta>0$ and $(x_0,t_0)\in\mathbb{R}^{n+1}_-$ we denote
\[
E_{\delta}(x_0,t_0):=\{(x,t)\in\mathbb{R}^{n+1}_-:\frac{1}{2}|x-x_0|^2-(t-t_0)< \delta,\quad t\le t_0 \},\quad E_{\delta}:=E_{\delta}(0,0).
\]
$\Delta_x:=\sum\limits_{i=1}^{n}D_{x_ix_i}$ is the Laplacian with respect to $x$ variable. For an invertible matrix $M$, its ‌inverse‌ is written as $\left(M^{ij}\right)$. The ‌parabolic norm‌ of $(x,t)\in \mathbb{R}^{n+1}_-$ is defined as $|(x,t)|_p:=\left(|x|^2-t\right)^{\frac{1}{2}}$. Assume the function $f$ satisfies the uniform bounds $\lambda\le f\le \Lambda$ in $\mathbb{R}^{n+1}_-$ for constants $\lambda,\Lambda>0$. For ease of narration, we assume that there exists $C_f>0$, such that for $i,j\in\mathbb{N}$ and $i+2j\le 2m$
\[
|D_x^iD_t^j\left(f-1\right)|\le C_f\mathcal{R}^{-\beta-(i+2j)} \quad \text{for} \quad \mathcal{R}\ge C_f.
\]

\section{a slow convergence solution to $\left(D_t-\Delta_x\right)u=f$} \label{a slow convergence solution}

In this section we construct an ancient solution with slow convergence to a type of inhomogenous heat equation.

\begin{lemma}\label{lemma1}
Let $f$ be a function defined in $\mathbb{R}^{n+1}_-$ satisfying that for $i,j\in \mathbb{N}$ with $i+2j\le3$, there is $D_x^iD_t^jf\in C^0\left(\mathbb{R}^{n+1}_-\right)$ and
\begin{equation}\label{festimate}
\left|D^i_xD_t^jf(x,t)\right|\le C_0 \mathcal{R}^{-\beta-i-2j}\quad \text{for}\quad \mathcal{R}\ge C_0,
\end{equation}
where $C_0>0$, $\beta\in(1,2]$ are constants. Then there exists $u\in C^{2,1}\left(\mathbb{R}^{n+1}_-\right)$ solving 
\begin{equation}\label{heatequation}
\left(D_t-\Delta_x\right)u=f\quad \text{in}\quad \mathbb{R}^{n+1}_-.
\end{equation}
Additionally there exists constant $C>0$ depending only on $n,C_0,\beta$ such that
\begin{equation}\label{slowestimate}
\left|u(x,t)\right|\le
\left\{\begin{array}{ll}
C\mathcal{R}^{2-\beta} &  \beta\in(1,2), \\
C\ln\mathcal{R} &  \beta=2,
\end{array}\right.\quad\text{for} \quad \mathcal{R}\ge C_0.
\end{equation}
\end{lemma}

\begin{proof}
\textbf{Step 1: Definition of $\mathbf{u}$. }Define for $1\le i\le n$
\[
\begin{aligned}
v_i(x,t)&:=\int_{\mathbb{R}^{n+1}_t}\Gamma(x-y,t-s)D_{x_i}f(y,s)dyds\\&=\int_{\mathbb{R}^{n+1}_-}\Gamma(y,-s)D_{x_i}f(x-y,t+s)dyds
\end{aligned}
\]
and 
\[
\begin{aligned}
v_0(x,t)&:=\int_{\mathbb{R}^{n+1}_t}\Gamma(x-y,t-s)D_tf(y,s)dyds\\&=\int_{\mathbb{R}^{n+1}_-}\Gamma(y,-s)D_{t}f(x-y,t+s)dyds,
\end{aligned}
\]
where
\[
\Gamma(x,t):=e^{-\frac{|x|^2}{4t}}\left(4\pi t\right)^{-\frac{n}{2}}
\]
is the heat kernal. From Lemma 2.4 in \cite{yb25}, we obtain that for $\mathcal{R}\ge C_0$
\begin{equation}\label{v_0v_i}
\left|v_0\right|\le C\mathcal{R}^{-\beta},\quad \left|v_i\right|\le C\mathcal{R}^{1-\beta}\quad 1\le i\le n.
\end{equation}

We define $u(x,t)$ as a line integral of the second type
\begin{equation}\label{lineintegral}
u(x,t):=\int_{L_{(0_n,0)\rightarrow(x,t)}}v_1dx_1+\cdots+v_ndx_n+v_0dt,
\end{equation}
where $0_{n}$ is the zero vector in $\mathbb{R}^n$ and $L_{\left(0_n,0\right)\rightarrow(x,t)}$ represents a piecewise smooth directed curve from $(x,t)$ to $\left(0_n,0\right)$ in $\mathbb{R}^{n+1}_-$. We will demonstrate that this integral is path-independent, ensuring $u(x,t)$ is well-defined. 

First we establish that
\begin{equation}\label{ij}
D_{x_j}v_i(x,t)=\int_{\mathbb{R}^{n+1}_-}\Gamma(y,-s)D_{x_ix_j}f(x-y,t+s)dyds,
\end{equation}
\begin{equation}\label{ti}
D_tv_i(x,t)=\int_{\mathbb{R}^{n+1}_-}\Gamma(y,-s)D_{x_it}f(x-y,t+s)dyds,
\end{equation}
\begin{equation}\label{it}
D_{x_i}v_0(x,t)=\int_{\mathbb{R}^{n+1}_-}\Gamma(y,-s)D_{x_it}f(x-y,t+s)dyds,
\end{equation}
and all these integrals are continous in $\mathbb{R}^{n+1}_-$. For $h\in\mathbb{R}^1$, we consider
\begin{equation}\label{differencequotient}
\begin{aligned}
D_{x_j}^hv_i(x,t)&=\int_{\mathbb{R}^{n+1}_-}\Gamma(y,-s)D_{x_ix_j}f(x+h\theta e_j-y,t+s)dyds\\
&=\int_{\mathbb{R}^{n+1}_t}\Gamma(x+h\theta e_j-y,t-s)D_{x_ix_j}f(y,s)dyds,
\end{aligned}
\end{equation}
where $D_{x_j}^h$ is the difference quotient operator, $\theta\in(0,1)$, and $\left\{e_1,\cdots,e_n\right\}$ is the orthonormal basis in $\mathbb{R}^n$. We separate $\mathbb{R}^{n+1}_t$ into two regions $\text{I}_1,\text{I}_2$, where
\[
\text{I}_1:=\left\{(x,t)\in\mathbb{R}^{n+1}_t:\left|(y,s)-(x,t)\right|_p\ge C_0\quad \text{and}\quad\left|(y,s)\right|_p\ge C_0\right\},
\]
\[
\text{I}_2:=\left\{(x,t)\in\mathbb{R}^{n+1}_t:\left|(y,s)-(x,t)\right|_p\le C_0\quad \text{or}\quad\left|(y,s)\right|_p\le C_0\right\}.
\]
Lemma 2.1 in \cite{yb25} establishes that for some constant $C(n)>0$
\[
\left|\Gamma(x,t)\right|\le C(n)\left|(x,t)\right|_p^{-n}.
\]
For $(y,s)\in\text{I}_1$ with $|h|\le\frac{C_0}{2}$, the triangular inequality yields
\[
\left|(x+h\theta e_j-y,t-s)\right|_p\ge \left|(y,s)-(x,t)\right|_p-|h|\ge\frac{\left|(y,s)-(x,t)\right|_p}{2},
\]
implying
\[
\left|\Gamma(x+h\theta e_j-y,t-s)D_{x_ix_j}f(y,s)\right|\le C\left|(y,s)-(x,t)\right|_p^{-n}\left|(y,s)\right|_p^{-\beta-2}.
\]
For $(y,s)\in\text{I}_2$ with $|h|\le\frac{C_0}{2}$, the boundedness of $D_{x_ix_j}f$ in $\mathbb{R}^{n+1}_t$ gives that 
\[
\left|\Gamma(x+h\theta e_j-y,t-s)D_{x_ix_j}f(y,s)\right|\le C\left|\left(y,s\right)-\left(x+h\theta e_j,t\right)\right|_p^{-n}.
\]
Hence the integrand in \eqref{differencequotient} can be controlled by a integrable function in $\mathbb{R}^{n+1}_t$. Applying the dominated convergence theorem
\[
\begin{aligned}
D_{x_j}v_i(x,t)=\lim\limits_{h\rightarrow0}D_{x_j}^hv_i(x,t)&=\int_{\mathbb{R}^{n+1}_t}\Gamma(x-y,t-s)D_{x_ix_j}f(y,s)dyds\\
&=\int_{\mathbb{R}^{n+1}_-}\Gamma(y,-s)D_{x_ix_j}f(x-y,t+s)dyds.
\end{aligned}
\]
The local uniform convergence follows from \eqref{festimate}.  Combined with the continouity of $D_x^iD_t^jf$ in $\mathbb{R}^{n+1}_-$ for $i,j\in \mathbb{N}$ with $i+2j\le3$, the continouity of $D_{x_j}v_i$ is proved, thereby establishing \eqref{ij}. Similar arguments verify \eqref{ti} and \eqref{it}.

From \eqref{ij}-\eqref{it}, we conclude that for $i,j\in\mathbb{N}$ and $1\le i,j\le n$
\[
D_{x_i}v_j(x,t)=D_{x_j}v_i(x,t),\quad D_{x_i}v_0(x,t)=D_tv_i(x,t).
\]
Thus the integral \eqref{lineintegral} is path-independent according to the Stokes's Theorem (see Theorem 10.33 in \cite{r76}).

\textbf{Step 2: Proof of \pmb{\eqref{heatequation}}.} Due to Lemma 2.3 in \cite{yb25}, we obtain that in $\mathbb{R}^{n+1}_-$
\[
\left(D_t-\Delta_x\right)v_i=D_{x_i}f,\quad \left(D_t-\Delta_x\right)v_0=D_tf.
\]
Thus for $1\le i\le n$
\[
D_{x_i}\left(\left(D_t-\Delta_x\right)u-f\right)=D_t\left(\left(D_t-\Delta_x\right)u-f\right)\equiv0,
\]
which means that
\[
\left(D_t-\Delta_x\right)u-f=\hat{C}\quad \text{in}\quad \mathbb{R}^{n+1}_-,
\]
where $\hat{C}$ is a constant. Notice that by \eqref{v_0v_i} $D_tu=v_0\rightarrow0$ as $\mathcal{R}\rightarrow\infty$. It also follows from \eqref{ij} that
\[
\begin{aligned}
\Delta_xu(x,t)&=\int_{\mathbb{R}^{n+1}_-}\Gamma(y,-s)\Delta_{x}f(x-y,t+s)dyds\\
&=\int_{\mathbb{R}^{n+1}_t}\Gamma(x-y,t-s)\Delta_{x}f(y,s)dyds.
\end{aligned}
\]
\eqref{festimate} gives that $\left|\Delta_xf\right|\le C\mathcal{R}^{-\beta-2}\quad \text{for}\quad \mathcal{R}\ge C_0$. From Lemma 2.4 in \cite{yb25}, $\Delta_xu(x,t)\rightarrow0$ as $\mathcal{R}\rightarrow\infty$. Therefore $\hat{C}=0$ and thus we have proved \eqref{heatequation}.

\textbf{Step 3: Proof of \pmb{\eqref{slowestimate}}.} We first observe the initial conditions and derivatives
\[
\begin{aligned}
u\left(0_n,0\right)=0,\quad D_{x_i}u(x,t)=v_i(x,t), \quad D_tu(x,t)=v_0(x,t).
\end{aligned}
\]
Applying the Fundamental Theorem of Calculus, we express $u$ as
\[
\begin{aligned}
u(x,t)&:=\int_0^1\frac{d}{ds}\left(u\left(sx,s^2t\right)\right)ds\\
&=\int_0^1D_xu\left(sx,s^2t\right)\cdot xds+\int_0^1D_tu\left(sx,s^2t\right)2stds.
\end{aligned}
\]
For $\mathcal{R}\ge C_0$, direct computation yields that 
\[
\begin{aligned}
\left|\int_0^1D_xu\left(sx,s^2t\right)\cdot xds\right|&\le \int_0^{C_0\mathcal{R}^{-1}}|x|ds+\int_{C_0\mathcal{R}^{-1}}^1\left(s\mathcal{R}\right)^{1-\beta}\mathcal{R}ds \\
&\le 
\left\{\begin{array}{ll}
C\mathcal{R}^{2-\beta} & \text { for } \beta\in(1,2), \\
C\ln\mathcal{R} & \text { for } \beta=2
\end{array}\right.
\end{aligned}
\]
and
\[
\begin{aligned}
\left|\int_0^1D_tu\left(sx,s^2t\right)2stds\right|&\le C\mathcal{R}^2\left( \int_0^{C_0\mathcal{R}^{-1}}sds+\int_{C_0\mathcal{R}^{-1}}^1\left(s\mathcal{R}\right)^{-\beta}sds\right) \\
&\le 
\left\{\begin{array}{ll}
C\mathcal{R}^{2-\beta} & \text { for } \beta\in(1,2), \\
C\ln\mathcal{R} & \text { for } \beta=2.
\end{array}\right.
\end{aligned}
\]
\end{proof}

\section{smoothing of solution} \label{smoothing of solution}
Given $C>0$, $m\in\mathbb{N}\setminus\left\{0\right\}$ and $u\in C^{2m,m}\left(\mathbb{R}^{n+1}\setminus\overline{E_C}\right)$, in order to apply Lemma \ref{lemma1}, we need to modify $u$ such that $u\in C^{2m,m}\left(\mathbb{R}^{n+1}_-\right)$. 

\begin{lemma}\label{smooth}
Let $r>0$, $m\in\mathbb{N}\setminus\left\{0\right\}$ and $u\in C^{2,1}\left(\mathbb{R}^{n+1}_-\right)\cap C^{2m,m}\left(\mathbb{R}^{n+1}_-\setminus\overline{E_r}\right)$ be a parabolically convex function satisfying \eqref{ut} and $D^2_xu\ge C_0I$ in $\mathbb{R}^{n+1}_-\setminus\overline{E_r}$ for $C_0>0$. Then there exists $\tilde{u}\in C^{2m,m}\left(\mathbb{R}^{n+1}_-\right)$ satisfying $\tilde{u}\equiv u$ in $\mathbb{R}^{n+1}_-\setminus\overline{E_{2r}}$, $\frac{m_1}{2}\le-\tilde{u}_t\le 2m_2$ and $D_x^2\tilde{u}\ge\delta_0 I$ in $\mathbb{R}^{n+1}_-$ for a constant $\delta_0>0$.
\end{lemma}

\begin{proof}
\textbf{Step 1: Definition and smoothness of $\mathbf{\tilde{u}}$.} Take $\delta\in\left(0,\frac{r}{2}\right)$ to be determined and let $\phi=\phi(x)$ be a radial smooth function supported on the ball $B_\delta$, such that $\int_{B_\delta}\phi=1$. We define the spatially mollified function
\[
\hat{u}(x,t)=\int_{B_\delta(x)}u(y,t)\phi(x-y)dy=\int_{B_\delta}u(x-y,t)\phi(y)dy.
\]
For $h>0$ to be determined, we define $\hat{u}_h^{[0]}:=\hat{u}$ and temporal regularization
\[
\hat{u}_h^{[i+1]}(x,t):=-\mkern-19mu\int_{t-h}^t\hat{u}_h^{[i]}(x,s)ds
\]
for $i\in\mathbb{N}$ and $i\ge0$. Next, construct a smooth cutoff function $\tilde{\eta}=\tilde{\eta}(s)$ satisfying
\[
\left\{\begin{array}{ll}
\tilde{\eta}\equiv1 & s\ge 2r, \\
\tilde{\eta}\equiv0 & 0\le s\le \frac{3}{2}r, \\
\left|\tilde{\eta}'\right|\le\frac{4}{r} & \frac{3}{2}r\le s\le 2r.
\end{array}\right.
\]
Define the space-time cutoff function $\eta(x,t):=\tilde{\eta}\left(\mathcal{R}(x,t)\right)$. The final regularized function is then constructed as
\[
\tilde{u}:=\eta u + \left(1-\eta\right)\left(\hat{u}_h^{[m-1]}+c|x|^2\right),
\]
which can be equivalently expressed as
\[
\tilde{u}=u+\left(1-\eta\right)\left(\hat{u}_h^{[m-1]}+c|x|^2-u\right)
\]
for $c>0$ to be determined. By the properties of $\phi$, $\hat{u}$ is smooth in $x$. Combining the fact that $\hat{u}_h^{[i]}$ is $i+1$-th differentiable continous in $t$ yields that $\tilde{u}\in C^{2m,m}\left(\mathbb{R}^{n+1}_-\right)$. 

\textbf{Step 2: Estimates of $\mathbf{\tilde{u}}$.} The convexity of $u$ in $x$ implies that 
\begin{equation}\label{convexity}
D_x^2\hat{u}\ge0 \quad\text{and}\quad D_x^2\hat{u}_h^{[i]}\ge0 \quad  \quad in \quad \mathbb{R}^{n+1}_-.
\end{equation}
Applying the Mean Value Theorem iteratively yields that
\begin{equation}\label{tderivative'}
\begin{aligned}
D_t\hat{u}_h^{[i]}(x,t)&=\frac{\hat{u}_h^{[i-1]}(x,t)-\hat{u}_h^{[i-1]}(x,t-h)}{h}\\
&=D_t\hat{u}_h^{[i-1]}\left(x,t-\theta_1h\right)\\
&\cdots\\
&=D_t\hat{u}\left(x,t-h\sum\limits_{j=1}^i\theta_j\right),
\end{aligned}
\end{equation}
where $\theta_1,\cdots,\theta_i\in(0,1)$. Hence we obtain the uniform bounds
\begin{equation}\label{strictmonotonicity}
m_1\le -D_t\hat{u}_h^{[i]}\le m_2 \quad \text{in}\quad \mathbb{R}^{n+1}_-.
\end{equation}

Here $\tilde{u}=u$ for $\mathcal{R}\ge 2r$ and the desired estimates for $D_x^2\tilde{u}$ and $\tilde{u}_t$ are obvious. For $\mathcal{R}\le \frac{3r}{2}$, there is $\tilde{u}=\hat{u}_h^{[m]}+c|x|^2$, which implies that $D^2_x\tilde{u}\ge 2cI$ from \eqref{convexity} and $m_1\le -D_t\tilde{u}\le m_2$ from \eqref{strictmonotonicity}. 

It remains to prove the estimates for $D_x^2\tilde{u}$ and $\tilde{u}_t$ when $\frac{3r}{2}\le \mathcal{R}\le 2r$. Direct calculations yield
\begin{equation}\label{Dtderivative}
D_t\tilde{u}=D_tu+D_t\left(1-\eta\right)\left(\hat{u}_h^{[m]}-u+c|x|^2\right)+\left(1-\eta\right)D_t\left(\hat{u}_h^{[m]}-u+c|x|^2\right)
\end{equation}
and
\begin{equation}\label{Dx2derivative}
\begin{aligned}
D_x^2\tilde{u}&=D_x^2u+D_x^2\left(1-\eta\right)\left(\hat{u}_h^{[m]}-u+c|x|^2\right)+\left(1-\eta\right)D_x^2\left(\hat{u}_h^{[m]}-u+c|x|^2\right)\\
&\quad+2D_x\left(1-\eta\right)\otimes D_x\left(\hat{u}_h^{[m]}-u+c|x|^2\right),
\end{aligned}
\end{equation}
where for $x,y\in\mathbb{R}^n$, we denote $x\otimes y$ as the matrix $\left(x_iy_j\right)$.

Now we estimate $D_x^iD_t^j\left(\hat{u}_h^{[m]}-u\right)$ for $i,j\in\mathbb{N}$ and $i+2j\le2$. The $C^{2,1}$ regularity of $u$ allows differentiation under the integral sign
\[
D_x^j\hat{u}_h^{[m]}(x,t)=\frac{1}{h^m}\int^t_{t-h}ds_1\int^{s_1}_{s_1-h}ds_2\cdots\int_{s_{m-1}-h}^{s_{m-1}}ds_m\int_{B_\delta}D_x^ju(x-y,s_m)\phi(y) dy,
\]
where $s_m\in\left(t-mh,t\right)$ and $j\in\mathbb{N}$ with $j\le2$. It follows that
\[
D_x^j\left(\hat{u}_h^{[m]}-u\right)(x,t)=\frac{1}{h^m}\int^t_{t-h}ds_1\cdots\int_{B_\delta}\left(D_x^iu(x-y,s_m)-D_x^ju(x,t)\right)\phi(y) dy.
\]
Due to \eqref{tderivative'}, we get that
\[
D_t\left(\hat{u}_h^{[m]}-u\right)(x,t)=\int_{B_\delta}\left(D_tu\left(x-y,t-h\sum\limits_{i=1}^m\theta_i\right)-D_tu(x,t)\right)\phi(y) dy
\]
Because of the uniform continouity of $D_tu$ and $D_x^ju$ in $\overline{E_{\frac{5r}{2}}}\setminus E_{r}$, for every $\epsilon>0$, we can choose $\delta,h$ small enough such that
\[
\left|D_x^ju(x-y,s_m)-D_x^ju(x,t)\right|<\epsilon,\quad \left|D_tu\left(x-y,t-h\sum\limits_{i=1}^m\theta_i\right)-D_tu(x,t)\right|<\epsilon
\]
for every $(x,t)\in E_{2r}\setminus\overline{E_{\frac{3r}{2}}}$ and $y\in B_\delta$. Then
\[
\left|D_x^j\left(\hat{u}_h^{[m]}-u\right)\right|<\epsilon,\quad \left|D_t\left(\hat{u}_h^{[m]}-u\right)\right|<\epsilon.
\]
Considering \eqref{Dtderivative} and \eqref{Dx2derivative}, we can choose $\epsilon,c$ small enough such that 
\[
D^2_x\tilde{u}\ge\frac{C_0}{2}I\quad \text{and} \quad\frac{m_1}{2}\le -D_t\tilde{u}\le 2m_2\quad \text{for} \frac{3r}{2}\le\mathcal{R}\le 2r. 
\]
Take $\delta_0:=\min\left\{2c,\frac{C_0}{2}\right\}$ and the lemma is proved.
\end{proof}

\quad

\section{proof of Theorem \ref{mainthm1}} \label{proof of Theorem 1}

According to Proposition 3.2 in \cite{yb25}, for fixed $\epsilon\in(0,1)$ small enough, there exist a $n\times n$ upper-triangular matrix $T$, $\tau>0$ with $\tau\det T^2=1$ and a constant $C>1$ depending only on $n,m_1,m_2,\lambda,\Lambda,\beta,C_f,\epsilon,\sup\limits_{B_1}u(\cdot,0)$, such that 
\[
C^{-1}I\le T\le CI,\quad m_1\le \tau\le m_2.
\] 
Moreover
\[
w(x,t):=u\left(T^{-1}x,\tau^{-1}t\right),\quad \ \eta(x,t):=w(x,t)-\frac{1}{2}|x|^2+t
\]
satisfy that
\[
\left|\eta(x,t)\right|\le C\mathcal{R}^{2-\epsilon}\quad \text{for}\quad \mathcal{R}\ge C.
\]

Through Proposition 4.1 in \cite{yb25} and Lemma \ref{smooth}, we establish that $u\in C^{2m,m}$ outside a bounded set from the regularity of $f$. By Lemma \ref{smooth}, we can mollify $u$ such that $u\in C^{2m,m}\left(\mathbb{R}^{n+1}_-\right)$ and
\[
D^2_xu\ge \delta_0 I\quad \text{and}\quad \frac{m_1}{2}\le-u_t\le 2m_2\quad \text{in} \quad\mathbb{R}^{n+1}_-
\]
for a constant $\delta_0>0$. Thus $f\in C^{2m-2,m-1}\left(\mathbb{R}^{n+1}_-\right)$ has positive lower and upper bounds in $\mathbb{R}^{n+1}_-$. We also have that $\eta,w\in C^{2m,m}\left(\mathbb{R}^{n+1}_-\right)$ and $\hat{f}\in C^{2m-2,m-1}\left(\mathbb{R}^{n+1}_-\right)$, where $
\hat{f}:=-w_t\det D^2_xw$.

\begin{proof}[Case for $0<\beta\le 1$]
For $\beta\in(0,1)$, following the proofs of Proposition 4.1-4.2 in \cite{yb25}, we establish that, there exists a constant $C_0>0$, such that
\[
\left|D_x^iD_t^j\eta(x,t)\right|\le C_0\mathcal{R}^{2-\beta-(i+2j)}\quad \text{for}\quad \mathcal{R}\ge C_0,
\]
where $i,j\in\mathbb{N}$ and $i+2j\le2m$. The case $\beta\in(0,1)$ is proved.

For $\beta=1$, Building upon Proposition 4.1-4.2 in \cite{yb25} and the regularity of $\eta$, we derive that for every $\epsilon_1\in\left(\frac{1}{2},1\right)$, there exist constants $C_1>0$, such that 
\begin{equation}\label{primaryest}
\left|D_x^iD_t^j\eta(x,t)\right|\le C_1\mathcal{R}^{2-\epsilon_1-(i+2j)}\quad \text{for} \quad \mathcal{R}\ge C_1,
\end{equation}
where $i,j\in\mathbb{N}$ and $i+2j\le2m$. It is easy to check that
\begin{equation}\label{Dxetaeq}
\begin{aligned}
\left( D_t-\Delta_x \right)D_x\eta &= \left( -D_tw\left( D_x^2w \right)^{ij}-\delta_{ij} \right)D_{x_ix_j}(D_x\eta)+\frac{D_twD_x\hat{f}}{\hat{f}} \\
&\quad=: g_1\quad \text{in}\quad \mathbb{R}^{n+1}_-,
\end{aligned}
\end{equation}
with $g_1$ having decay properties
\[
\left|D_x^iD_t^jg_1(x,t)\right| \le C \mathcal{R}^{-2-i-2j}\quad \text{for}\quad \mathcal{R}\ge C_1,
\]
where $i,j\in\mathbb{N}$ and $i+2j\le3$. By Lemma \ref{lemma1}, there exists $\phi_1\in C^{2,1}\left(\mathbb{R}^{n+1}_-\right)$ solving
\begin{equation}\label{phi1eq}
\left(D_t-\Delta_x\right)\phi_1=g_1\quad \text{in} \quad\mathbb{R}^{n+1}_-
\end{equation}
and the estimates
\begin{equation}\label{geC2}
\left|\phi_1(x,t)\right|\le C\ln\mathcal{R}\quad \text{for}\quad \mathcal{R}\ge C_1.
\end{equation}
Thus by \eqref{Dxetaeq} and \eqref{phi1eq}, there is 
\[
\left(D_t-\Delta_x\right)\left(D_x\eta-\phi_1\right)=0_n\quad \text{in}\quad \mathbb{R}^{n+1}_-.
\] 
\eqref{primaryest} and \eqref{geC2} imply that
\begin{equation}\label{eq01}
\left|D_x\eta-\phi_1\right|\le C\mathcal{R}^{1-\epsilon_1}\quad \text{for}\quad \mathcal{R}\ge C_1.
\end{equation}
It is easy to see that
\begin{equation}\label{eq02}
\left|D_x\eta-\phi_1\right|\le C\quad \text{for}\quad \mathcal{R}\le C_1.
\end{equation}
Therefore 
\[
\left|D_x\eta-\phi_1\right|\le C\left(\mathcal{R}+1\right)^{1-\epsilon_1}\quad \text{in}\quad \mathbb{R}^{n+1}_-.
\]
Combined with Theorem 1.2(b) in \cite{lz19}, we obtain that $D_x\eta-\phi_1$ is independent of $t$ variable and harmonic in $\mathbb{R}^n$. It follows that
\[
\left|D_x\eta-\phi_1\right|\le C\left(|x|+1\right)^{1-\epsilon_1}\quad \text{in}\quad \mathbb{R}^n.
\]
By spherical harmonic expansions (see for example Lemma 1 in \cite{lb1} for $n\ge3$ and Lemma 3.1 in \cite{lb2} for $n=2$), we obtain that when $n\ge2$
\[
\left|D_x\eta-\phi_1\right|\le C\quad \text{in}\quad \mathbb{R}^n. 
\]
Notice that for $n=1$, the above estimate is direct. 
By \eqref{geC2} we attain that 
\[
\left|D_x\eta\right|\le C\ln \mathcal{R}\quad \text{for}\quad \mathcal{R}\ge C_1.
\] 

To estimate $D_t\eta$, we notice that
\[
\left( D_t-\Delta_x \right)(D_t\eta) = \left( -D_tw\left( D_x^2w \right)^{ij}-\delta_{ij} \right)D_{x_ix_j}(D_t\eta) + \frac{D_twD_t\hat{f}}{\hat{f}} =: g_2\quad \text{in}\quad \mathbb{R}^{n+1}_-
\]
and 
\[
\left|g_2(x,t)\right|\le C\mathcal{R}^{-3}\quad \text{for} \quad \mathcal{R}\ge C_1. 
\]
We define
\[
\phi_2(x,t):=\int_{\mathbb{R}^{n+1}_t}K(x-y,t-s)g_2(y,s)dyds \quad \text{for} \quad (x,t)\in \mathbb{R}^{n+1}_-.
\]
Using Lemma 2.3-2.4 in \cite{yb25}, we get that 
\[
\left(D_t-\Delta_x\right)\left(D_t\eta-\phi_2\right)=0\quad \text{in}\quad \mathbb{R}^{n+1}_-
\]
and
\[
\left|\phi_2\right|\le C\mathcal{R}^{-1}\quad \text{for}\quad \mathcal{R}\ge C_1.
\]
Combining \eqref{primaryest} for $i=0,j=1$ and Lemma 2.5-2.6 in \cite{yb25}, we conclude that
\[
\left|D_t\eta-\phi_2\right|\le Ce^{\frac{-\sqrt{2}\mathcal{R}}{8}}\quad \text{for}\quad \mathcal{R}\ge C_1
\]
and then
\[
\left|D_t\eta\right|\le C\mathcal{R}^{-1}\quad \mathcal{R}\ge C_1.
\]

Now we estimate $\eta$. For $\mathcal{R}\ge C_2$, there is
\[
\begin{aligned}
\eta(x,t)-\eta\left(0_n,0\right)&=\int_0^1\frac{d}{ds}\left( \eta\left(sx,s^2t\right) \right)ds\\&=\int_0^1D_x\eta\left(sx,s^2t\right)\cdot xds+2t\int_0^1D_t\eta\left(sx,s^2t\right)sds,
\end{aligned}
\]
where
\[
\begin{aligned}
&\left|\int_0^1D_x\eta\left(sx,s^2t\right)\cdot xds\right|\\
&=\int_0^{C^2\mathcal{R}^{-1}}\left|D_x\eta\left(sx,s^2t\right)\right||x|ds+\int_{C^2\mathcal{R}^{-1}}^1\left|D_x\eta\left(sx,s^2t\right)\right||x|ds\\
&\le C\left( \int_0^{C_2\mathcal{R}^{-1}}C\mathcal{R}ds + \int_{C^2\mathcal{R}^{-1}}^1\ln\left(s\mathcal{R}\right)\mathcal{R}ds \right)\le C\mathcal{R}\ln\mathcal{R}
\end{aligned}
\]
and
\[
\begin{aligned}
2|t|\left|\int_0^1D_t\eta\left(sx,s^2t\right)sds\right|\le C\mathcal{R}^{2}\left( \int_0^{C_2\mathcal{R}^{-1}}sds + \int_{C_2\mathcal{R}^{-1}}^1\left(s\mathcal{R}\right)^{-1}sds \right)\le C\mathcal{R}.
\end{aligned}
\]
As a result 
\[
\left|\eta(x,t)\right|\le C\mathcal{R}\ln\mathcal{R}\quad \text{for} \quad\mathcal{R}\ge C_1. 
\]
Following the proof of Proposition 4.1 in \cite{yb25},  we obtain that
\[
\left|D_x^iD_t^j\eta(x,t)\right|\le C\left(\ln\mathcal{R}\right)\mathcal{R}^{1-(i+2j)}\quad \text{for}\quad \mathcal{R}\ge 2C_1,
\]
where $i,j\in\mathbb{N}$ and $i+2j\le2m$. This completes the proof for $\beta=1$.
\end{proof}

\begin{proof}[Case for $1<\beta\le 2$]
For $\beta\in(1,2)$, following the proof of Proposition 4.4 in \cite{yb25}, we derive that there exist $\hat{b}\in\mathbb{R}^n$ and $C_2>0$, such that
\[
\left|D_x^iD_t^j\hat{\eta}\right|\le C_2\mathcal{R}^{2-\beta-(i+2j)}\quad \text{for}\quad \mathcal{R}\ge C_2,
\]
where 
\[
\hat{\eta}(x,t):=\eta(x,t)-\hat{b}\cdot x,
\] 
$i,j\in\mathbb{N}$ and $i+2j\le 2m$. The case $\beta\in(1,2)$ has been proved.

For $\beta=2$, performing the routine in the proof of Proposition 4.4 in \cite{yb25} and fixing $\epsilon_1\in\left(\frac{1}{2},1\right)$, we attain that there exist constants $C_3>0$, such that
\begin{equation}\label{hat}
\left|D_x^iD_t^j\hat{\eta}\right|\le C_3\mathcal{R}^{2-2\epsilon_1-(i+2j)}\quad \text{for}\quad \mathcal{R}\ge C_3,
\end{equation}
where $i,j\in\mathbb{N}$ and $i+2j\le2m$. It follows from the equation of $w$ that
\[
\left( D_t-\Delta_x \right)\hat{\eta} = \left( -\frac{\tilde{a}_{ij}}{\tilde{a}_1} - \delta_{ij} \right)D_{x_ix_j}\hat{\eta} + \frac{\log\hat{f}}{\tilde{a}_1}=:g_3\quad \text{in}\quad \mathbb{R}^{n+1}_-,
\]
where
\begin{equation}\label{coefficient}
\tilde{a}_1(x,t):=\int_0^1\left(  sD_t\hat{\eta}-1 \right)^{-1}ds,\quad
\tilde{a}_{ij}(x,t):=\int_0^1\left( sD_x^2 \hat{\eta}+I\right)^{ij}ds.
\end{equation}
It is easy to check that 
\[
\left|D_x^iD_t^jg_3(x,t)\right| \le C \mathcal{R}^{-2-i-2j}\quad \text{for}\quad \mathcal{R}\ge C_3,
\]
where $i,j\in\mathbb{N}$ and $i+2j\le3$. Using Lemma \ref{lemma1}, there exists $\phi_3\in C^{2,1}\left(\mathbb{R}^{n+1}_-\right)$ satisfying the equation
\[
\left(D_t-\Delta_x\right)\left(\hat{\eta}-\phi_3\right)=0\quad \text{in}\quad \mathbb{R}^{n+1}_-
\]
and the estimates
\begin{equation}\label{phi3}
\left|\phi_3\right|\le C\ln\mathcal{R}\quad \text{for}\quad \mathcal{R}\ge C_3,
\end{equation}
It follows from \eqref{hat} and \eqref{phi3} that
\[
\left|\hat{\eta}-\phi_3\right|\le C\mathcal{R}^{2-2\epsilon_1}\quad \text{for} \quad \mathcal{R}\ge C_3.
\]
Thus we observe that
\[
\left|\hat{\eta}-\phi_3\right|\le C\left(\mathcal{R}+1\right)^{2-2\epsilon_1}\quad \text{in} \quad \mathbb{R}^{n+1}_-.
\]
Theorem 1.2(b) in \cite{lz19} gives that $\hat{\eta}-\phi_3$ is a function independent of $t$ and is harmonic in $\mathbb{R}^n$. It follows that
\[
\left|\hat{\eta}-\phi_3\right|\le C\left(|x|+1\right)^{2-2\epsilon_1}\quad \text{in}\quad \mathbb{R}^n.
\]
By spherical harmonic expansions, we obtain that 
\[
\left|\hat{\eta}-\phi_3\right|\le C\quad \text{in} \quad\mathbb{R}^n. 
\]
Then we attain that 
\[
\left|\hat{\eta}\right|\le C\ln \mathcal{R}\quad \text{for} \quad\mathcal{R}\ge C_3. 
\]
Following the proof of Proposition 4.1 in \cite{yb25} gives that
\[
\left|D_x^iD_t^j\hat{\eta}(x,t)\right|\le C\left(\ln\mathcal{R}\right)\mathcal{R}^{2-(i+2j)}\quad\text{for}\quad \mathcal{R}\ge 2C_3,
\]
where $i,j\in\mathbb{N}$ and $i+2j\le2m$. This completes the proof for $\beta=2$.
\end{proof}

Denote $b:=T'\hat{b}$, $A:=T'T$. Theorem \ref{mainthm1} is proved after scaling back.

\noindent {\bf Funding:} J. Bao is supported by the National Natural Science Foundation of China (12371200).

\medskip

\bibliographystyle{plain}
\def\cprime{$'$}

\end{document}